\documentclass[12pt]{amsart}

\usepackage{amsfonts,amsmath,amsthm}
\usepackage{latexsym}

\newtheorem{theorem}{Theorem}[section]

\newtheorem{proposition}[theorem]{Proposition}
\newtheorem{definition}{Definition}[section]

\theoremstyle{definition}

\newcommand{\ZZ}{\ensuremath{\mathbb{Z}}}
\newcommand{\RR}{\ensuremath{\mathbb{R}}}

\newcommand{\ca}{\ensuremath{\mathcal{A}}}

\newcommand{\vd}{\ensuremath{\mathbf{d}}}
\newcommand{\vb}{\ensuremath{\mathbf{b}}}

\newcommand{\vu}{\ensuremath{\mathbf{u}}}

\newcommand{\vw}{\ensuremath{\mathbf{w}}}
\newcommand{\vv}{\ensuremath{\mathbf{v}}}
\newcommand{\vz}{\ensuremath{\mathbf{z}}}

\newcommand{\one}{\ensuremath{\mathbf{1}}}

\def \< {\langle}
\def \> {\rangle}

\begin{document}

\title[Finite projective planes and the Delsarte LP-bound]{Finite projective planes and the Delsarte LP-bound}

\author[M. Matolcsi]{M\'at\'e Matolcsi}
\address{M.M.: Budapest University of Technology and Economics (BME),
H-1111, Egry J. u. 1, Budapest, Hungary (also at Alfr\'ed R\'enyi Institute of Mathematics,
Hungarian Academy of Sciences, H-1053, Realtanoda u 13-15, Budapest, Hungary)}
\email{matomate@renyi.hu}

\author[M. Weiner]{Mih\'aly Weiner}
\address{M.W.: Budapest University of Technology and Economics (BME),
H-1111, Egry J. u. 1, Budapest, Hungary}
\email{mweiner@math.bme.hu}

\thanks{M. Matolcsi was supported by the ERC-AdG 321104 and by NKFIH Grant No. K109789, M. Weiner was supported by the ERC-AdG 669240, and by NKFIH Grant No. K124152.}

\begin{abstract}
We apply an improvement of the Delsarte LP-bound to give a new proof of the non-existence of finite projective planes of order 6, and uniqueness of finite projective planes of order 7. The proof is computer aided, and it is also feasible to apply to higher orders like 8, 9 and, with further improvements, possibly 10 and 12.
\end{abstract}

\maketitle

\bigskip

\section{Introduction}

In this note we apply the Delsarte LP-bound (and a small improvement of it) to the problem of existence of finite projective planes. The method is computer aided and we have carried it out for orders $n=6$ and 7. It is still feasible for other small orders such as $n=8, 9$ (although $n=9$ probably requires computation on a cluster and an increased running time). For higher orders, like 10 and 12, it is likely that further ideas are needed to exploit certain invariance properties of the problem in order to reduce the running time and make the approach feasible.

\medskip

For the sake of self-containment we include all the necessary notions in the Introduction, while we describe the results in Section \ref{results}. We begin by the definition of finite affine planes, projective planes, and complete sets of mutually orthogonal Latin squares (MOLs).

\begin{definition}
A finite affine plane of order $n$ is a collection $\ca=\{P_1, P_2, \dots, P_{n^2}\}$ of $n^2$ points, together with $n+1$ parallel classes\\
 $L_1, L_2, \dots, L_{n+1}$ on $\ca$ with the following properties: each $L_i$ is a collection $L_i=\{\ell_i^{(0)}, \ell_i^{(1)}, \dots, \ell_i^{(n-1)}\}$ of $n$ distinct parallel (i.e. non-intersecting) lines each containing $n$ points of $\ca$, such that for each $i\ne j$ the intersection of any two lines $\ell_i^{(k)}$ and $\ell_j^{(m)}$ is exactly one point.
\end{definition}

The definition of a finite projective plane is usually given in an intrinsic manner, referring only to incidences of lines and planes. However, for the sake of unified notations we prefer to give here an equivalent definition based on finite affine planes.

\begin{definition}
A finite projective plane of order $n$ is the disjoint union of a finite affine plane $\ca=\{P_1, P_2, \dots, P_{n^2}\}$ and a "line at infinity" $\ell_\infty=\{P_1^{(\infty)}, P_2^{(\infty)}, \dots, P_{n+1}^{(\infty)}\}$. We define $P_i^{(\infty)}$ to be the intersection of any two distinct lines $\ell_i^{(k)}$ and $\ell_i^{(m)}$ of the parallel class $L_i$ of the affine plane $\ca$.
\end{definition}

\begin{definition}
A Latin square $S=[s_{i,j}]_{i,j=0}^{n-1}$ of order $n$ is an $n\times n$ square filled out with symbols $0, 1, \dots, n-1$ such that each row and each column contains each symbol exactly once. Two Latin squares $S_1$, $S_2$ are called orthogonal if the ordered pairs $(s_{i,j}^{(1)}, s_{i,j}^{(2)})$ are all distinct (i.e. $(s_{i_1,j_1}^{(1)}, s_{i_1,j_1}^{(2)})\ne (s_{i_2,j_2}^{(1)}, s_{i_2,j_2}^{(2)})$ for any $(i_1,j_1)\ne (i_2, j_2)$, where $s_{i,j}^{(k)}$ denotes the entry of Latin square $S_k$ at position $(i,j)$). A complete set of mutually orthogonal Latin squares (MOLs) is a collection $S_1, S_2, \dots, S_{n-1}$ of $n-1$ pairwise orthogonal Latin squares.
\end{definition}

\medskip

It is well-known that the existence of these objects are all equivalent, but we give short proof of this fact in Proposition \ref{equiv} below, because we will need the construction appearing in the proof.

\medskip

If $n$ is a prime-power, finite projective planes of order $n$ can be constructed using finite fields. If $n$ is not a prime-power, it is widely believed that finite projective planes of order $n$ do not exist. Over 100 years ago Tarry \cite{tarry} proved that there exist no two orthogonal Latin squares of order 6, which implies the nonexistence of a finite projective plane of order $6$. His proof is based on a rather tedious checking of each $6\times 6$ Latin square. Some 40 years later, Bruck and Ryser \cite{bruckryser} proved their celebrated result that if a finite projective plane of order $d \equiv 1,2$ mod$(4)$ exists, then $d$ must be a sum of two squares. This result rules out an infinite family of non-primepower orders (including 6, again), but leaves the problem open for orders such as $d=10$ or $d=12$. As of today, for $d=10$ we only know the nonexistence because of a massive computer search \cite{LamThiefSwiercz}, and for $d=12$, the question is still open. In this paper we present an approach based on Delsarte's LP-bound, which can successfully be applied to small orders ($n=6, 7$), and possibly also to higher ones ($8\le n\le 12$) in the future. A different approach based on the non-commutative version of the Delsarte scheme was presented recently \cite{ncMol} by the authors for $n=6$.

\medskip

After these historical remarks, we turn to the question of equivalence of the notions defined above.

\begin{proposition}\label{equiv}
The existence of a finite affine plane of order $n$, a finite projective plane of order $n$, and a complete set of MOLs of order $n$ are all equivalent.
\end{proposition}

\begin{proof}
It is clear from the definition that the existence of finite affine planes and finite projective planes of order $n$ are equivalent (just add a line at infinity to an affine plane to get a projective plane, or remove any particular line from a projective plane to get an affine plane).

\medskip

Given an affine plane $\ca$ of order $n$ we can construct a complete set of MOLs as follows. First, consider the parallel classes $L_n$ and $L_{n+1}$ of $\ca$ as "vertical" and "horizontal" lines, respectively. This introduces a coordinate system on the points of $\ca$: let a point $P\in \ca$ be identified with the pair $(i,j)$ if and only if $P$ is the intersection of $\ell_n^{(i)}$ and $\ell_{n+1}^{(j)}$. Note that each pair $0\le i,j\le n-1$ corresponds to exactly one point of $\ca$. After this, the parallel class $L_k$ will give rise to a Latin square $S_k$ (for $1\le k\le n-1$) in the following natural way: for any $0\le m\le n-1$ put the symbol $m$ in $S_k$ to the entries $(i,j)$ which belong to the line $\ell_k^{(m)}$. By the properties of the affine plane $\ca$ it is easy to see that each $S_k$ becomes a Latin square, and $S_{k_1}$ and $S_{k_2}$ are orthogonal if $k_1\ne k_2$. Hence, we have constructed a complete set of MOLs.

\medskip

Given a complete set $S_1, S_2, \dots, S_{n-1}$  of MOLs of order $n$, the construction of a finite affine plane $\ca$ of order $n$ is completely analogous: the points of $\ca$ will be identified with coordinates $(i,j)$, the positions of a particular symbol $0\le m\le n-1$ in $S_k$ will define the line $\ell_k^{(m)}$, and the parallel class $L_k$ will be given as the collection of lines $\ell_k^{(m)}$ for $0\le m\le n-1$. Finally, the parallel classes $L_n$ and $L_{n+1}$ will be given as the vertical lines $\ell_n^{(m)}=\{(m,0), (m,1), \dots, (m, n-1)\}$ and the horizontal lines $\ell_{n+1}^{(m)}=\{(0,m), (1,m), \dots (n-1, m)\}$, respectively.
\end{proof}

\medskip

We will also need the fact that the existence of an affine plane of order $n$ is equivalent to the existence of a set of $n^2$ elements of 
$\ZZ_n^n\equiv \{0, 1, \dots, n-1\}^n$ of minimal Hamming distance $n-1$ (where $\ZZ_n$ denotes the cyclic group of order $n$). For the sake of
self-containment (and because it seems difficult to find a straightforward reference), we shall formally state and prove this equivalence, too.

\begin{proposition}\label{cor1}
Let $B\subset \ZZ^n_n$ be a collection of
$n^2$ elements of $\ZZ^n_n$ such that
$\vv-\tilde{\vv}$ has at most one coordinate equalling $0$ (i.e.\! such that $\vv$ and $\tilde{\vv}$ has at most one coinciding coordinate value) for all $\vv,\tilde{\vv}\in B, \vv\neq\tilde{\vv}$. Then $B$ can be partitioned into $n$ classes (disjoint subsets) of size $n$ such that $\vv-\tilde{\vv}$ has precisely one coordinate equalling to $0$ whenever $\vv,\tilde{\vv} \in B$ belong to different classes, and no coordinate equalling to $0$ whenever they are different elements of the same class. As such, the collection $B$ is naturally equivalent to an affine plane
of order $n$.
\end{proposition}
\begin{proof}
Let $B\subset \ZZ^n_n$ be a collection of size $n^2$ with the described property
and choose a pair $j<k$ of indices. Then for any $z_j,z_k\in \ZZ_n$ there can be at most one vector of $B$ whose $j^{\rm th}$ and $k^{\rm th}$ coordinates are $z_j$ and $z_k$, respectively. But as the pair $(z_j,z_k)$ can take $n^2$ different values and $|B|=n^2$, this means that there is {\it precisely one} such vector of $B$. It follows that each value of $\ZZ_n$ appears altogether $n^2$ times in the vectors of $B$.

\medskip

Let us pick now a vector $\vv\in B$. An application of a permutation of $\ZZ_n$ at the $j^{\rm th}$ coordinate of each vector of $B$ obviously preserves the required property of the collection, thus we may assume without loss of generality that $\vv=(0,0,\ldots 0)$.  Let $\mathbf w$ be a vector containing no 0s. Again, because of permutation invariance we may assume that $\mathbf w=(1,1,\ldots 1)$. Pick a vector $\vu\ne \vv\in B$ which contains no 1s. We will show that $\vu$ automatically contains no 0s, either.

\medskip

Each of the $n^2-2$ vectors of $B\setminus\{\vv,\mathbf w\}$ can contain at most one $0$ and one $1$. However, for each $j\neq k$ there is exactly one vector in $B$ whose $j^{\rm th}$ coordinate is $0$ and $k^{\rm th}$ coordinate is $1$, and these vectors are all different (since none of them can contain more $0$s or $1$s). Hence,  we have an $n(n-1)$-element subset $B_{0,1}$ of $B$ with each vector containing both a $0$ and a $1$. Therefore, $B_{0,1}\cup \{\vv, \mathbf w\}$ accounts for all the $n^2$ 0s and 1s appearing in $B$, and hence the remaining $n-2$ vectors of $B$ must be free of $0$s and 1s. As $\vu\notin B_{0,1}\cup \{\vv, \mathbf w\}$, it can only be one of the remaining $n-2$ vectors, and it is therefore free of 0s. 

\medskip

This argument shows in general two important facts. One is that for any vector $\vv\in B$ we have $n-1$ other vectors of $B$ with no coinciding coordinate with $\vv$. The other is that if $\vv,\vw$ and $\vu$
are $3$ different vectors of $B$ and the pair $(\vv,\vw)$ and
$(\vw,\vu)$ have no coinciding coordinates, then also $\vv$ and $\vu$ can have
no coinciding coordinates. From here it is clear that $B$ must be a disjoint union of classes in the manner described in the proposition.

\medskip

We now prove the equivalence of $B$ with an affine plane $\ca$ of order $n$. Assume $\ca$ is given. As in the construction of Proposition \ref{equiv}, associate a complete set of MOLs to the affine plane $\ca$. For any $1\le k\le n-1$ and $0\le m\le n-1$ the line $\ell_k^{(m)}$ corresponds to the positions $(i,j)$ of symbol $m$ in $S_k$. Note that each row and column of $S_k$ intersects $\ell_k^{(m)}$ exactly once. Let us order the points $(i,j)$ of $\ell_k^{(m)}$ according to their fist coordinate: $\ell_k^{(m)}=\{(0, j_0), (1, j_1), \dots, (n-1,j_{n-1})\}$, where $(j_0, j_1, \dots, j_{n-1})$ is automatically a permutation of the numbers $0, 1, \dots, n-1$. Let $\vv_k^{(m)}=(j_0, j_1, \dots, j_{n-1})$. Finally, for $k=0$ let $\vv_k^{(m)}=(m, m, \dots, m)$.

\medskip

By construction it is easy to see that $\vv_{k}^{(m_1)}-\vv_{k}^{(m_2)}$ contains no 0 coordinate: for $k=0$ it is obvious, and for $1\le k\le n-1$ the vectors correspond to positions of different symbols $m_1$ and $m_2$ in $S_k$. It is also easy to see that for any $k_1<k_2$ the vector $\vv_{k_1}^{(m_1)}-\vv_{k_2}^{(m_2)}$ contains exactly one 0 coordinate: it is obvious if $k_1=0$ (because $\vv_{k_2}^{(m_2)}$ is a permutation), and for $0<k_1<k_2$ the reason is that the lines $\ell_{k_1}^{(m_1)}$ and $\ell_{k_2}^{(m_2)}$ have exactly one point of intersection.

\medskip

Assume now that $B\subset \ZZ_n^n$ is given. We will construct an affine plane $\ca$ of order $n$. The points of $\ca$ will be given as elements of $\ZZ_n^2$. To each vector $\vu=(u_0, \dots, u_{n-1}) \in B$ associate the set of elements $\{(0,u_0), \dots, (n-1, u_{n-1})\}$. This will be considered as a line $\ell_\vu$ of $\ca$. Two such lines $\ell_\vu, \ell_\vv$ will be parallel if $\vu$ and $\vv$ have no coinciding coordinates, and have one point of intersection if $\vu$ and $\vv$ have one coinciding coordinate. Hence, from the first part of the proof it is clear that this gives rise to $n$ parallel classes of lines such the intersection of any two lines in different classes is exactly one point. Finally, we need to augment this construction with the set of "vertical" lines $\ell^{(m)}=\{(m,0), (m,1), \dots, (m, n-1)\}$, for $0\le m\le n-1$, and we obtain the affine plane $\ca$.
\end{proof}

\medskip

We now turn to the description of the particular form of the Delsarte LP-bound that we will be using. Given a finite Abelian group $(G, +)$ and a symmetric ``forbidden set'' $A=-A\subset G$, at most how many elements a subset $B=\{b_1,\ldots b_n\}\subset G$ can have, if all differences $b_j-b_k$ ($j\neq k$) ``avoid'' $A$ (i.e. fall into $A^c$)? This is a very general type of question and many famous problems can be re-phrased in this manner. (In some applications $G$ is not finite, e.g. $G=(\RR^n, +)$ but we will only consider finite groups in this note.)

\medskip

A method that has often proved fruitful when dealing with such problems
--- e.g.\! in the context of sphere-packing \cite{cohnelkies, sp8} or in the maximum number of code-words in error correcting codes \cite{delsarte} ---
is based on the observation (used originally by Delsarte in \cite{delsarte}) that the function $1_B\ast 1_{-B}$ is positive definite over $G$. A Fourier-analytic formulation of this general scheme over finite Abelian groups was described in \cite{int1}. We invoke the relevant result here.

\begin{theorem}[\cite{int1}]\label{dels}
Let $(G, +)$ be a finite Abelian group, $A=-A\subset G$ be a symmetric subset containing 0, and $B\subset G$ be a subset such that $b-b'\notin A$ for all $b\ne b'\in B$. Assume there exists a function $f: G\to \RR$ such that $f|_{A^c}\le 0$, $\hat{f}(\gamma)=\sum_{x\in G}\gamma(x)f(x)\ge 0$ for all characters $\gamma$ of $G$. Then
\begin{equation}\label{db}
|B|\le \frac{f(0)|G|}{\hat{f}(\one)},
\end{equation}
where $\one$ denotes the constant 1 character.
\end{theorem}
\begin{proof}
This is the inequality $\delta (A)\le \lambda^- (A)$ in Theorem 1.4 in \cite{int1}.
\end{proof}

\medskip

The problem of existence of finite affine planes fits into this Delsarte-scheme as follows: let $G=\ZZ_n^n$ and let the "forbidden set" $A$ be given as the set of vectors containing at least two coordinates equalling 0. In order to conclude the non-existence of a finite affine (or, equivalently, projective) plane of order $n$ it is sufficient to show that the maximum number of vectors in $G$ such that all differences avoid $A$, is strictly less than $n^2$. Unfortunately, the Delsarte LP-bound described in Theorem \ref{dels} implies only that the number of such vectors is $\le n^2$ (see Proposition \ref{dbound} below). However, we will be able to invoke ideas from \cite{impdel} where a small improvement of the Delsarte LP-bound is given. In particular, \cite[Theorem 2]{impdel} gives a concrete numerical improvement of the bound $|B|\le \frac{f(0)|G|}{\hat{f}(\one)}$  under specific circumstances. However, for the sake of clarity and self-containment, we will not cite the general result \cite[Theorem 2]{impdel} verbatim, but rather adapt the idea of its proof to the present situation in Proposition \ref{impbound}.

\section{Results}\label{results}

In this section we first apply Delsarte's LP-bound directly to the situation described above, and then make a small improvement to conclude non-existence of finite projective planes of order 6, and uniqueness of that of order 7.

\medskip

We aim to apply the bound \eqref{db}. Let $G=\ZZ_n^n$, and let $\omega=e^{2i\pi/n}$. For the sake of simpler notation in subsequent formulas, it is more convenient to change perspective and think of a multiplicative structure on $G$ rather than an additive one.  Namely, think of elements of $\ZZ_n$ as powers of $\omega$, with the operation on $\ZZ_n$ being multiplication. In this formalism an element of $G$ can be identified with a vector $\vz=(z_1, z_2, \dots, z_n)=(\omega^{j_1}, \omega^{j_2}, \dots, \omega^{j_n})$ with exponents $0\le j_i\le n-1$. The operation on $G$ is multiplication coordinate-wise. The identity element of $G$ is the vector $(1, 1, \dots, 1)$. The characters $\gamma$ of $G$ are functions of the form $\gamma(\vz)=\gamma(z_1, z_2, \dots, z_n)=z_1^{g_1}z_2^{g_2}\dots z_n^{g_n}$ for any choice of exponents $0\le g_i \le n-1$. Also, let the "forbidden set" $A$ be given as
\begin{equation}\label{a}
A=\{\vz=(z_1, z_2, \dots, z_n)\in G: \textrm{at least two coordinates equal 1}\}.
\end{equation}

\begin{proposition}\label{dbound}
Let $G=\ZZ_n^n$ be given with the multiplicative structure described in the previous paragraph. Let $B\subset G=\ZZ_n^n$ be a subset such that each quotient $\vb/\vb'$ ($\vb\ne \vb'\in B$) contains zero or one coordinate equalling 1. Then $|B|\le n^2$.
\end{proposition}
\begin{proof}
Consider the function $f: G\to \RR$ given by the following formula:
\begin{equation}\label{tanu}
f(\vz)=f(z_1, z_2, \dots, z_n)=\left (\sum_{i=1}^n\sum_{j=0}^{n-1}z_i^j \right )\left (-n+\sum_{i=1}^n\sum_{j=0}^{n-1}z_i^j \right )
\end{equation}
It is immediate that $f$ vanishes on $A^c$ (the first term is 0 if none of the $z_i$ equal 1, and the second term is 0 if exactly one of the $z_i$ equals 1). It is also easy to see that $f$ is a polynomial with nonnegative coefficients, and hence that $\hat{f}(\gamma)\ge 0$ for all $\gamma$ (note here that $\hat{f}(\gamma)$ is exactly the coefficient of $z^\gamma$ times $|G|$). Also, $\hat{f}(\one)/|G|$ is the constant term of $f$, which equals $n(n-1)$. Furthermore, $f(1, 1, \dots, 1)=n^3(n-1)$. Therefore, equation \eqref{db} implies $|B|\le n^2$ as desired.
\end{proof}

\medskip

The statement of Proposition \ref{dbound} is quite trivial, as can be seen by easy combinatorial arguments. We included the proof above mainly to illustrate how the Delsarte LP-bound can be applied to this situation.

\medskip

The bound of Proposition \ref{dbound} is sharp whenever $n$ is a prime-power (simply because an affine plane of order $n$ exists, and Proposition \ref{cor1} provides a collection of $n^2$ suitable vectors). Also, it is not hard to prove that for any value of $n$ the function $f$ above is best possible in the sense that $n^2$ is the smallest possible value on the right hand side of \eqref{db}. Therefore, the Delsarte LP-bound \eqref{db}, in itself, is not sufficient to prove non-existence results for any order $n$. However, the ideas of \cite{impdel} can be invoked to get a small improvement on the Delsarte-bound. We will do so by applying the main idea of \cite{impdel} to this particular situation, rather than explicitly referring to the general results described in \cite{impdel}.

\medskip

\begin{proposition}\label{zerosum}
Let $G=\ZZ_n^n$ be given with the multiplicative structure as in the previous proposition. Assume $B\subset G=\ZZ_n^n$ is a set of $n^2$ vectors such that each quotient $\vb/\vb'$ (for $\vb\ne \vb'\in B$) contains exactly zero or one coordinate equalling 1. Let $z_0=1$ be a dummy variable (for convenience of notation), and for $0\le i<j\le n$, $0\le k, m\le n-1$, let $f_{i,j}^{k,m}(\vz)=f_{i,j}^{k,m}(z_1, z_2, \dots, z_n)=z_i^kz_j^m$. Then for $(k,m)\ne (0,0)$ we have
\begin{equation}\label{b}
\sum_{\vb\in B}f_{i,j}^{k,m}(\vb)=0
\end{equation}
\end{proposition}
\begin{proof}
Let $1\le i<j\le n$ and consider the ordered pairs $(b_i, b_j)$ formed by the $i$th and $j$th coordinates of the vectors $\vb\in B$. Obviously, all these pairs are distinct, otherwise a quotient $\vb/\vb'$ would have coordinates equalling 1 at position $i$ and $j$ (note here that the pairs $(b_i, b_j)$ being distinct gives an elementary proof of Proposition \ref{dbound}) . Also, there are $n^2$ such pairs because $|B|=n^2$. This means that the pairs $(b_i, b_j)$ must exhaust the set of all pairs $\{(\omega^0,\omega^0), (\omega^0,\omega^1), \dots, (\omega^{n-1},\omega^{n-1})\}$, each pair of exponents appearing exactly once. Hence, an elementary calculation shows that  $\sum_{\vb\in B}f_{i,j}^{k,m}(\vb)=0$.

\medskip

If $0=i<j$, then $f_{i,j}^{k,m}(\vz)=z_j^m$. Consider the $j$th coordinates of the vectors of $B$. There are $n^2$ such numbers, and from the argument above we see that each power of $\omega$ appears exactly $n$ times. An elementary calculation shows again that  $\sum_{\vb\in B}f_{i,j}^{k,m}(\vb)=0$.
\end{proof}

We hope to exploit Proposition \ref{zerosum} if some of the vectors appearing in $B$ are already known.

\begin{proposition}\label{impbound}
Let $G=\ZZ_n^n$, and $A\subset G$ be the "forbidden set" defined in \eqref{a}. Assume $B_0=\{\vv_1, \vv_2, \dots, \vv_s\}\subset G=\ZZ_n^n$ is given. Let $D\subset G$ denote the set of vectors $\vd\in G$ such that $\vd/\vv\in A^c$ for all $\vv\in B_0$. Assume that there exists a function $h: G\to \RR$ such that $h=\sum_{i,j,k,m}\lambda_{i,j,k,m}f_{i,j}^{k,m}$ is a linear combination of the functions $f_{i,j}^{k,m}$ appearing in the previous proposition, $\sum_{\vv\in B_0} h(\vv)=1$, and $h(\vd)\ge 0$ for all $\vd\in D$. Then the set $B_0$ cannot be extended to a set $B$ of $n^2$ vectors such that the quotients $\vb/\vb'$ fall into $A^c$ for all $\vb\ne \vb'\in B$.
\end{proposition}
\begin{proof}
Assume by contradiction that an extension $B$ of $B_0$ exists such that $|B|=n^2$ and $\vb/\vb'\in A^c$ for all $\vb\ne \vb'\in B$. By the definition of $D$, any vector $\vv\in B\setminus B_0$ must belong to $D$. Therefore, $h(\vv)\ge 0$ for all $\vv\in B\setminus B_0$. This implies that
\begin{equation}\label{d}
\sum_{\vb\in B}h(\vb)=\sum_{\vv\in B_0}h(\vv)+\sum_{\vv\in B\setminus B_0}h(\vv)\ge 1,
\end{equation}
which contradicts equation \eqref{b}, because $h$ is a linear combination of the functions $f_{i,j}^{k,m}$, so the sum on the left hand side of \eqref{d} should be 0.
\end{proof}

\medskip

Proposition \ref{impbound} gives us a tool to prove non-existence and uniqueness results concerning finite projective planes. The idea is that there is only a restricted number of ways to fix the first few vectors of $B$, after which we can hope to arrive at a contradiction by finding a suitable witness function $h$ as in Proposition \ref{impbound}. As finding a function $h$ with the required properties involves solving a linear programming problem, it is most conveniently done by computer. We have documented this procedure for $n=6$ and 7 (see below), but it is still feasible for $n=8, 9$ (although $n=9$ probably requires computation on a cluster and an increased running time). As it stands now, the running time definitely gets out of proportion for $n\ge 10$, and further ideas are needed to make the search conclusive.

\begin{theorem}\label{main}
There exist no finite affine (or, equivalently, projective) plane of order 6. The projective plane of order 7 is unique up to isomorphism.
\end{theorem}

\begin{proof}
Assume that a finite affine plane of order $n$ exists. Then the construction of Proposition \ref{cor1} produces a collection of $n^2$ vectors $B=\{\vv_k^ {(m)}: 0\le k, m\le n-1\}\subset G$ such that for any $k_1\ne k_2$ and any $m_1, m_2$ the vector $\vv_{k_1}^{(m_1)}-\vv_{k_2}^{(m_2)}$ contains exactly one coordinate equaling 0, while for any $k$ and any $m_1\ne m_2$ the vector  $\vv_{k}^{(m_1)}-\vv_{k}^{(m_2)}$ contains no coordinates equaling 0. In fact, the constant vectors $(m, m, \dots, m)$ for $0\le m\le n-1$ appear automatically in $B$, by construction. Let us also select those non-constant vectors $\vv_1, \vv_2, \vv_{n-1}$ in $B$ whose first coordinate is 0. If consider these vectors as columns of length $n$, and place them one after the other, we obtain a matrix of size $n\times (n-1)$, with first row equaling constant 0. By the properties of the set $B$, if we delete the first row of 0's of the matrix, the remaining matrix is automatically a Latin square of size $n-1$. Therefore, we can assume without loss of generality that $B$ contains the constant vectors $(m, m, \dots, m)$ for $0\le m\le n-1$, and some further vectors $\vv_1, \vv_2, \dots, \vv_{n-1}$ whose first coordinate is 0, and whose other coordinates form a Latin square $S$ of size $n-1$. It is also clear that $S$ can be chosen as any representative of an isotopy class of the Latin squares of size $n-1$ (i.e., we are free to permute rows, columns and symbols in $S$). After fixing $S$, $2n-1$ vectors of $B$ are already given. We let $B_0$ denote the set of these vectors, and hope to arrive at a contradiction by applying Proposition \ref{impbound} to testify that $B_0$ cannot be extended to a full set of $n^2$ vectors $B$.

\medskip

For $n=6$ there are only 2 isotopy classes $S_1^{(5)}, S_2^{(5)}$  of Latin squares (see \cite{web}) of order $n-1=5$. Hence, there are essentially only two different ways of picking $B_0$, the first 11 vectors of our hypothetical set $B$. A short computer code \cite{code} then testifies that a suitable witness function $h$ (as described in Proposition \ref{impbound}) can be found for $B_0$ corresponding to $S_2^{(5)}$, proving that $B_0$ cannot be extended to a full set of 36 vectors in this case. In the case of $S_1^{(5)}$, the set of candidate vectors $D$ (as defined in Proposition \ref{impbound}) contains 75 vectors, and an appropriate witness function $h$ does not exist. However, if we pick any of those 75 vectors to be further included in $B_0$, a witness function already exists in all 75 cases. This concludes the proof of non-existence of finite projective planes of order 6.

\medskip

For $n=7$ there are only 22 isotopy classes $S_1^{(6)}, S_2^{(6)}, \dots, S_{22}^{(6)}$ of Latin squares (see \cite{web}) of order $n-1=6$. Hence, there are essentially only 22 different ways of picking $B_0$, the first 13 vectors of our hypothetical set $B$. A short computer code \cite{code} then testifies that a suitable witness function $h$ (as described in Proposition \ref{impbound}) can be found in 19 of these cases (the exceptional cases being $S_1^{(6)}, S_2^{(6)}$ and $S_4^{(6)}$. In two of the cases (for $S_1^{(6)}$ and $S_4^{(6)}$)  the set of candidate vectors $D$ consists of 288 and 216 vectors, respectively, and for all choices of those vectors a suitable witness function already exists. The only remaining case $S_2^{(6)}$ leads to the unique affine plane of order 7.
\end{proof}

Let us make a few concluding remarks. The success of the approach depends essentially on two factors: the number of isotopy classes of Latin squares $S$ of order $n-1$, and whether a suitable witness function (as described in Proposition \ref{impbound}) can typically be found once a representative $S$ is fixed. Unfortunately, the number of isotopy classes grows very fast. Also, as $n$ increases we encounter many cases when a suitable witness function $h$ simply does not exist after fixing $S$. In such a case, further vectors must be added to $B_0$ until a suitable function $h$ is finally found. However, adding further vectors to $B_0$ means branching out the search space, and increasing the running time. We have not made a full documentation for $n=8$, but we did test the 564 different Latin squares of size 7 as choices for $S$. In 230 of those cases a witness function $h$ exists, while in the remaining 334 cases further vectors must be added to $B_0$. Judging upon this, we estimate the running time of the full algorithm on a single PC to be a few days for $n=8$.
For $n=9$ we tend to believe that the running time could still be reasonable if the computations are made on a cluster. However, some new ideas seem to be necessary to cover any case $n\ge 10$. For this we note that there is a fair amount of flexibility in the method: there are many ways to go about choosing the first few vectors $B_0$ of $B$, and one does not need to stick to the idea of Latin squares of order $n-1$ described above. Other selections of $B_0$ may exploit the symmetries of the problem better, and hence lead to a much reduced running time.

\begin{center}
\large{Acknowledgement}
\end{center}

The authors are grateful for the reviewer for several suggestions to improve the quality of the paper.

\end{document}